\documentclass[10pt]{amsart}
\usepackage{amsmath}
\usepackage{amssymb}
\usepackage{xypic}
\usepackage[a4paper]{geometry}
\def\ba{\mathbf A}

\def\A{{\mathbf A}}
\def\C{{\mathbb C}}

\def\N{{\mathbf N}}
\def\G{{\mathbf G}}

\def\Jac{\mathrm{Jac}}

 \newtheorem*{theorem*}{Theorem}
  \theoremstyle{plain}
  
\theoremstyle{plain}
\newtheorem{theorem}{Theorem}[section]
  \theoremstyle{definition}
  
  \newtheorem{notation}[theorem]{Notation}   
  \theoremstyle{plain}
  \newtheorem{proposition}[theorem]{Proposition}
  \theoremstyle{plain}
  \newtheorem{lemma}[theorem]{Lemma}
   \theoremstyle{definition}
   \newtheorem{remark}[theorem]{Remark}

\begin{document}
\author{Adrien Dubouloz, Lucy Moser-Jauslin and Pierre-Marie Poloni} 
\address{Adrien Dubouloz and Lucy Moser-Jauslin\\ CNRS\\ Institut de Math\'{e}matiques de Bourgogne\\ Universit\'{e} de Bourgogne\\ 9 Avenue Alain Savary\\ BP 47870\\ 21078 Dijon Cedex\\ France} \email{Adrien.Dubouloz@u-bourgogne.fr} \email{moser@u-bourgogne.fr}
\address{ Pierre-Marie Poloni\\ Mathematisches Institut Universit\"at Basel Rheinsprung 21 CH-4051 Basel Switzerland\\}
\email{pierre-marie.poloni@unibas.ch}
\thanks{This research was supported in part by the ANR Grant BirPol ANR-11-JS01-004-01.}

\title{Automorphism groups of certain rational hypersurfaces in complex four-space}

\begin{abstract}
The Russell cubic is a smooth contractible affine complex threefold which is not isomorphic to affine three-space. In previous articles,  we discussed the structure of the automorphism group of this variety. Here we review some consequences of this structure and generalize some results to other hypersurfaces which arise as deformations of Koras-Russell threefolds.
\end{abstract}
\maketitle

 \section{Introduction}

In order to prove the linearizability of algebraic actions of $\C^*$ on affine three-space, \cite{K-R,K-K-ML-R}, Koras and Russell studied hyperbolic $\C^*$-actions on more general smooth contractible threefolds. This led them to  introduce a set of threefolds which are smooth affine and contractible, however not isomorphic to $\ba^3$. These varieties are now known as Koras-Russell threefolds.  One of the families of these varieties, called Koras-Russell threefolds of the first kind, is given by hypersurfaces $X_{d,k,\ell}$ in the affine space $\mathbb{A}^4=\mathrm{Spec}(\mathbb{C}[x,y,z,t])$ defined by equations of the form  $x^dy+z^k+t^\ell+x=0$ where $d\geq 2$ and $2\le k<\ell$ with $k$ and $\ell$ relatively prime. All of these threefolds admit algebraic actions of the complex additive group $\mathbb{G}_a$ and they were originally proven to be not isomorphic to affine space by means of invariants associated to those actions. These invariants, known as the Derksen and Makar-Limanov invariants, are defined respectively for an affine variety $X=\mathrm{Spec}(A)$ admitting non trivial $\mathbb{G}_a$-actions as the sub-algebra $\mathrm{Dk}(X)$ of $A$ consisting of regular functions invariant under \emph{at least} one non trivial $\mathbb{G}_a$-action on $X$ and its sub-algebra $\mathrm{ML}(X)$ consisting of regular functions invariants under \emph{all} non trivial such actions.

These tools have since become important and useful to study affine algebraic varieties. In particular, one of the central elements in the proofs of many existing results concerning Koras-Russell threefolds of the first kind, and some of the generalizations we consider in this article, is the fact that their Makar-Limanov and Derksen invariants are equal to $\mathbb{C}[x]$ and $\mathbb{C}[x,z,t]$ respectively (see, for example, \cite{K-ML1997}, Lemma 8.3 and \cite{K-ML2}, Example 9.1.  for the Koras-Russell threefolds).  This property imposes very strong restrictions on the nature of isomorphisms between such varieties which enable sometimes an explicit description of their isomorphism classes and automorphism groups.

Koras-Russell threefolds of the first kind belong to the more general
family of hypersurfaces $X=X(d,r_{0},g)$ in $\mathbb{A}^{4}=\mathrm{Spec}(\mathbb{C}[x,y,z,t])$
defined by equations of the form
\[
x^{d}y+r_{0}(z,t)+xg(x,z,t)=0
\]
where $d\geq2$, $r_{0}\in\mathbb{C}[z,t]$ and $g\in\mathbb{C}[x,z,t]$.
All varieties of this type share the property that they come equipped with a flat
$\mathbb{A}^{2}$-fibration $\pi=\mathrm{pr}_{x}:X\rightarrow\mathbb{A}^{1}=\mathrm{Spec}(\mathbb{C}[x])$
restricting to a trivial bundle over the complement of the origin
and with degenerate fiber $\pi^{-1}(0)$ isomorphic to the cylinder
$C_{0}\times\mathbb{A}^{1}$ over the plane curve $C_{0}\subset\mathrm{Spec}(\mathbb{C}[z,t])$
with equation $r_{0}=0$. In particular, noting that $\pi^{-1}(\A^1\setminus\{0\})$ is factorial, and that $\pi^{-1}(0)=\mathrm{div}(x)$ is a prime principal divisor if and only if $C_0$ is reduced and irreducible, we see that a threefold $X$ is factorial
whenever the corresponding curve $C_{0}$ is reduced and irreducible (see also \cite{Nag}). 
A combination of \cite{Kal} and \cite{Sath} implies that $X$ is isomorphic to $\mathbb{A}^3$
if and only if $\pi^{-1}(0)$ is reduced and isomorphic to $\mathbb{A}^2$, whence, by virtue of \cite{AEH}
if and only if $C_0$ is isomorphic to the affine line. 
Furthermore, identifying the coordinate ring $A$ of $X$ with the
sub-algebra $\mathbb{C}[x,z,t,x^{-d}(r_{0}+xg(x,z,t))]$ of $\mathbb{C}[x,z,t]_{x}$
via the canonical localization homomorphism with respect to $x$ gives
rise to a description of $X$ as the affine modification $\sigma=\mathrm{pr}_{x,z,t}\mid_{X}:X\rightarrow\mathbb{A}^{3}$
of $\mathbb{A}^{3}=\mathrm{Spec}(\mathbb{C}[x,z,t])$ with center
at the closed subscheme $Z$ with defining ideal $J=(x^{d},r_{0}(z,t)+xg(x,z,t))$
and divisor $D=\{x^{d}=0\}$ in the sense of \cite{K-Z99}. That is, $X$ is isomorphic to the complement of
the proper transform of $D$ in the blow-up of $\mathbb{A}^{3}$ with
center at $Z$.
Noting that the closed subscheme $Z$ of $\mathbb{A}^{3}$ is supported along the curve $C_{0}\subset\mathrm{Spec}(\mathbb{C}[z,t])$, this description implies that a smooth $X$ for which $C_{0}$ is irreducible, topologically contractible, but not isomorphic to the affine line, is an exotic $\mathbb{A}^{3}$ \cite[Theorem 3.1]{K-Z99}. This holds for instance for smooth deformations
of Koras-Russell threefolds of the first kind defined by equations
of the form $x^{d}y+z^{k}+t^{\ell}+xg(x,y,z)=0$, with $k,\ell\geq2$
relatively prime and $g(0,0,0)=1$, corresponding to the irreducible,
singular, topologically contractible plane curves $C_{0}=\{z^{k}+t^{\ell}=0\}$.

The present article reviews three complementary  applications of Derksen and Makar-Limanov invariants to the study
of threefolds $X(d,r_0,g)$ as above. First we summarize
several properties of automorphism groups of Koras-Russell threefolds
of the first kind which appeared separately in previous articles
by the authors, and we complete the picture
with a characterization of certain natural subgroups of these automorphism
groups. Then we turn to the study of
non-necessarily smooth
deformations
of Koras-Russell threefolds defined by equations of the form $x^{d}y+z^{k}+t^{\ell}+xg(x,y,z)=0$.
We explain how to obtain a description of isomorphism classes of these
threefolds that is reminiscent of the (mini)-versal deformation of
the corresponding singular plane curve $C_{0}=\{z^{k}+t^{\ell}=0\}.$
Finally, we illustrate on an example of a threefold $X=\{x^{d}y+r_{0}(z,t)=0\}$
with non-connected associated plane curve $C_{0}=\{r_{0}=0\}$ a general
procedure to construct new types of automorphisms of $X$ which do not admit
any extension to automorphisms of the ambient space $\mathbb{A}^{4}$.

\section{A preliminary observation}

Let $d\in\N$ and $r_0\in\C[z,t]$ be fixed. For any $g\in\C[x,z,t]$, we denote by $$J_g=(x^d,r_0+xg)$$ the ideal of $\C[x,z,t]$ generated by $x^d$ and $r_0+xg$, and  by $A(g)$ the coordinate ring of the hypersurface $X(g)$ of $\ba^4=\mathrm{Spec}(\C[x,y,z,t])$ defined by the equation $$x^dy+r_0(z,t)+xg(x,z,t)=0.$$ Corresponding to the presentation of $X(g)$ as the affine modification $\sigma : X(g)\rightarrow \mathbb{A}^3$ mentioned in the introduction, we have a chain of inclusions $$\C[x,z,t]\subset A(g)\subset A(g)[x^{-1}]\simeq \C[x,x^{-1},z,t].$$  The second inclusion is induced by the localization homomorphism with respect to the regular element $x\in A(g)$,  identifying $y\in A(g)$ with $-x^{-d}(r_0+xg)\in \C[x,x^{-1},z,t]$.

Given a pair of polynomials $f,g\in \C[x,z,t]$, the universal property of affine modifications \cite[Proposition 2.1]{K-Z99} implies that every automorphism $\varphi$ of $\C[x,z,t]$ which fixes the ideal $(x)$ and maps $J_g$ into $J_f$ lifts in a unique way to a morphism $\tilde{\varphi}:A(g)\rightarrow A(f)$ restricting to $\varphi$ on the subring $\C[x,z,t]$. Actually, $\tilde{\varphi}$ is even an isomorphism. Indeed, by hypothesis, there exist $\alpha\in\C^*$ and $a,b\in\C[x,z,t]$ such that $\varphi(x)=\alpha x$ and  $\varphi(r_0+xg)=ax^d+b(r_0+xf)$. The second equation implies that $b$ is congruent modulo $x$ to a non-zero constant whence that its residue class in $\C[x,z,t]/(x^d)$ is a unit for every $d\geq 1$. Choosing $b'\in\C[x,z,t]$ such that $bb'\equiv 1 \mod (x^d)$ and multiplying the previous equation by it, we conclude that $r_0+xf\in\varphi(J_g)$. Thus $\varphi$ maps $J_g$ isomorphically onto $J_f$.

The following lemma will be used several times throughout this article.

\begin{lemma} \label{rem} With the notation above assume further that the Derksen and Makar-Limanov invariants of $X(f)$ and $X(g)$ are  equal to $\C[x,z,t]$ and $\C[x]$, respectively. Then the previous construction provides a one-to-one correspondance between isomorphisms from  $A(g)$ to $A(f)$ and automorphisms of $\C[x,z,t]$ which fix the ideal $(x)$ and map $J_g$ into $J_f$.
\end{lemma}

\begin{proof}
Note first that the hypotheses imply that neither $X(f)$ nor $X(g)$ is isomorphic to $\mathbb{A}^3$ and hence, as a consequence of \cite{Sath}, that the surface $C_{0}\times \mathbb{A}^1=\{r_{0}=0\}\subset\mathrm{Spec}(\mathbb{C}[x,z,t])$ is not isomorphic to $\mathbb{A}^2$. Since an isomorphism between $A(g)$ and $A(f)$ preserves the Makar-Limanov  and the Derksen invariants, it restricts to an automorphism $\varphi$ of $\C[x,z,t]$ and an automorphism of $\C[x]$. That is,  $\varphi(x)$ is of the form $ax+b$  where $a\in \C^*$ and $b\in \C$. Actually, $b=0$ since by the previous remark the zero set of $ax+b$ in $X(f)$ and $X(g)$ is non isomorphic to $\mathbb{A}^2$ if and only if $b=0$. This shows that $\varphi$ fixes the ideal $(x)$. Noting that $J_g=x^dA(g)\cap \C[x,z,t]$ and similarly for $J_f$, we conclude that $\varphi(J_g)=J_f$, which completes the proof.
\end{proof}

\section{Automorphisms of Koras-Russell threefolds of the first kind}

In this section, we consider  Koras-Russell threefolds of the first kind  $X=X(d,k,\ell)$ corresponding to the cases where $d\ge 2$, $r_0=z^k+t^\ell$, and $g=1$. Since $\mathrm{Dk}(X)=\C[x,z,t]$ and $\mathrm{ML}(X)=\C[x]$, we deduce from Lemma \ref{rem} that the projection $\sigma=\mathrm{pr}_{x,z,t}\mid_X: X \rightarrow \mathbb{A}^3$ gives rise to an isomorphism between the automorphism group of $X$ and the subgroup $\mathcal{A}$ of automorphisms of $\C[x,z,t]$ which preserve the ideals $(x)$ and $(x^d, r_0+x)$. In particular, $\C^*$ acts linearly on $X$. In fact, letting $\mathcal{A}_n$, $1\leq n\leq d$ be the normal subgroup of $\mathcal{A}$ consisting of the automorphisms $\varphi$ which fix $x$ and which are congruent to the identity modulo $(x^n)$, it was shown more precisely in \cite{d-mj-p,mj09} that $$\mathrm{Aut}(X)\simeq \mathcal{A}_1\rtimes \C^* \quad \textrm{and} \quad \mathcal{A}_n/\mathcal{A}_{n+1}\cong (\C[z,t],+)\quad\textrm{for~all~}1\le n\le d-1.$$

\noindent The next proposition summarizes several consequences of this description:

\begin{proposition} Let $X=X_{d,k,\ell}\subset\ba^4$ be a Koras-Russell threefold of the first kind. Then the following hold:

\noindent \; $1)$ Every  automorphism of $X$ extends to an automorphism of $\ba^4$.

\noindent \; $2)$ The group $\mathrm{Aut}(X)$ acts on $X$ with exactly 4 orbits:

\quad  - an open orbit $\{x\neq 0\}\simeq\C^*\times \C^2$,

\quad  - a copy of $\C^*\times\C$ given by $x=0$ and $z\not=0$,

\quad  - the line $\{x=z=t=0\}$ minus the point $(0,0,0,0)$,

\quad  - a fixed point $(0,0,0,0)$.

\noindent \; $3)$ Every finite subgroup of $\mathrm{Aut}(X)$ is cyclic.

\noindent \; $4)$ Every one-parameter unipotent subgroup of $\mathrm{Aut}(X)$ is contained in $\mathcal{A}_d$. In particular, the subgroup generated by all $\G_a$-actions on $X$ is strictly smaller than $\mathcal{A}_1$.

\end{proposition}

\begin{remark} In contrast with Property 2) above, the group of holomorphic automorphisms of $X$ acts with at most three orbits. Indeed, one checks for instance that the holomorphic automorphism $\Psi$ of $\ba^4$ defined by
  $$\Psi(x,y,z,t)=(x,\textrm{e}^{x^{d-1}}y-\frac{1-\textrm{e}^{x^{d-1}}}{x^{d-1}},\textrm{e}^{\frac{x^{d-1}}{k}}z,\textrm{e}^{\frac{x^{d-1}}{\ell}}t)$$
maps $X$ onto itself. Hence $\Psi$ induces a holomorphic automorphism $\psi$ of $X$ for which $\psi(0,0,0,0)=(0,1,0,0)$, i.e. $(0,0,0,0)$ is no longer a fixed point for the action of holomorphic automorphisms of $X$. The exact number of orbits under the action this group is not known. In particular, it is still an open question whether any threefold $X_{d,k,\ell}$ is biholomorphic to the affine space.
\end{remark}

\begin{proof}
Properties 1) and 2) were established in \cite{d-mj-p} for the Russell cubic and in \cite{mj09} for the general case.

The third property was originally formulated as a question by V. Popov. Since $\mathrm{Aut}(X)\simeq \mathcal{A}_1\rtimes \C^*$, it is enough to show that $\mathcal{A}_1$ does not contain non-trivial torsion elements. Indeed, if so, the projection to the second factor $\C^*$ will induce an isomorphism between every finite subgroup of $\mathrm{Aut}(X)$ and a subgroup of $\C^*$. So suppose that $\varphi\in\mathcal{A}_1$ is a nontrivial torsion element, say of order $m\geq 2$. By possibly switching $z$ and $t$, we can further assume that $\varphi(z)\not=z$. Choosing $N\in \N$ minimal with the property that $\varphi(z)\equiv z+f(z,t)x^N \mod (x^{N+1})$ for some $f\in\C[z,t]\setminus\{0\}$, we would have $\varphi^m(z)\equiv z+mf(z,t)x^N\not\equiv z \mod (x^{N+1})$, a contradiction.

For the last property, it follows from Lemma \ref{rem} that every additive group action on $X$ is induced by a locally nilpotent derivation $D$ of the coordinate ring of $X$ extending a locally nilpotent $\C[x]$-derivation of $\C[x,z,t]$ which maps the ideal $J=(x^d,z^k+t^\ell+x)$ into itself, the second condition being equivalent to the property that the corresponding exponential automorphisms preserve this ideal. We will show that in fact the image of $D$ is contained in the ideal $(x^d)$, which implies that the corresponding one-parameter unipotent subgroup of $\mathrm{Aut}(X)$ is contained in $\mathcal{A}_d$. We prove this by induction, assuming that $D\equiv 0$ modulo $(x^k)$ with $0\le k<d$. Since $D(z^k+t^\ell+x)=ax^d+b(z^k+t^\ell+x)$ where $a,b\in\C[x,z,t]$, the hypothesis implies that $x^k$ divides $b$. On the other hand, $D_1=x^{-k}D$ is again a locally nilpotent derivation such that $D_1(x)=0$ and for which we have $D_1(z^k+t^\ell+x)=ax^{d-k}+(b/x^k)(z^k+t^\ell+x)$. Thus $D_1$ induces a locally nilpotent derivation $\overline{D_1}$ of $\C[x,z,t]/(x)\cong\C[z,t]$ which maps the ideal generated by $z^k+t^\ell$ into itself. So $\overline{D_1}$ is necessarily trivial as there is no nontrivial $\G_a$-action preserving a singular plane curve. Thus $D_1\equiv 0$ modulo $(x)$ and hence $D\equiv 0 \mod (x^{k+1})$.
\end{proof}

\begin{remark}  Recall that we always have an exact sequence of groups $$ 0\rightarrow \mathrm{Aut}_0(\ba^4,X)\rightarrow \mathrm{Aut}(\ba^4,X)\stackrel{\rho}{\rightarrow}\mathrm{Aut}(X)$$ where $\mathrm{Aut}(\ba^4,X)$ denotes the subgroup of $\mathrm{Aut}(\ba^4)$ consisting of automorphisms which leave $X$ invariant and where $\mathrm{Aut}_0(\ba^4,X)$ denotes the kernel of $\rho$. The surjectivity of $\rho$ was established in Property 1) of the above proposition by constructing explicit lifts of automorphisms of $X$ to automorphisms of $\ba^4$. Nevertheless, this construction was only set-theoretic and it is not clear whether the above sequence splits.
Note however that since an element $\varphi\in \mathcal{A}_d$ is the identity modulo $(x^d)$ and preserves the subring $\C[x,z,t]$ of the coordinate ring of $X$, it lifts in a natural way to an automorphism $\Phi$ of $\C[x,y,z,t]$ by letting simply $\Phi(y)=y+(\varphi(r_0+x)-r_0-x)/x^d$. This gives  rise to group homomorphism $\mathcal{A}_d\rightarrow \mathrm{Aut}(\ba^4,X)$, $\varphi\mapsto \Phi$, which, combined with the fact that the action of $\C^*$ on $X$ comes as the restriction of a linear action on $\ba^4$, extends to a group homomorphism $j:\mathcal{A}_d\rtimes \C^*\rightarrow \mathrm{Aut}(\ba^4,X)$ such that $\rho \circ j=\mathrm{id}$. We do not know whether $j$ can be extended to a splitting of the above exact sequence.
 \end{remark}

\section{Deformations of Koras-Russell threefolds }
In this section, we consider hypersurfaces  $X(g)$  of $\ba^4=\mathrm{Spec}(\C[x,y,z,t])$ defined by equations of the form  $$x^dy+r_0(z,t)+xg(x,z,t)=0$$ where $r_0=z^k+t^\ell$ and $d\geq 2$ are fixed, and we let the polynomial $g\in\mathbb{C}[x,z,t]$ vary. The case $r_0=z^2+t^3$ was treated in \cite{d-mj-p2}, leading to the construction of large families of non-isomorphic smooth affine threefolds that are all biholomorphic to each other and diffeomorphic to the affine space. Here, we show that similar techniques can be applied to find isomorphisms between deformations of hypersurfaces in a more general setting.

Theorem \ref{MainTh} below relies again in a crucial way on the fact that the Derksen and Makar-Limanov invariant of threefolds $X(g)$ are equal to $\mathbb{C}[x,z,t]$ and $\mathbb{C}[x]$ respectively. These properties can be checked using the methods developed in \cite{K-ML2}, namely via a careful study of homogeneous locally nilpotent derivations on a well-chosen quasi-homogeneous deformation of $X(g)$. The complete proof is quite long and technical but is essentially straightforward using the aforementioned methods. We shall omit it here, in particular since it involves no new ideas or arguments.

\begin{notation}
 We will be considering derivations of $\C[z,t]$ defined by the Jacobian of a polynomial. If $f\in\C[z,t]$, we denote by $f_z$ the partial derivative $\partial f/\partial z$, and by $f_t$ the  partial derivative $\partial f/\partial t$. For $f,g\in \C[z,t]$, the Jacobian $\Jac(f,g)$ denotes $f_zg_t-f_tg_z$. Finally, $\Jac(f,\cdot)$ denotes the derivation of $\C[z,t]$ defined by $g\mapsto \Jac(f,g)$.
\end{notation}

Let $B=\C[a_{i,j}]$ be the polynomial ring in the $(k-1)(\ell-1)$ indeterminates $a_{i,j}$, $0\leq i\leq k-2, 0\leq j\leq \ell -2$, let $\mathfrak{m}_0\subset B$ be the maximal ideal generated by the $a_{i,j}$ and let $$F=r_0+\sum_{0\leq i\leq k-2,\, 0\leq j\leq \ell -2} a_{i,j} z^it^j \in B[z,t].$$
Given a homomorphism $\alpha :B \rightarrow \C[x]$ such that $\alpha(\mathfrak{m}_0)\subset x\C[x]$, the image of $F$ in $\C[x]\otimes_B B[z,t]\simeq \C[x,z,t]$ has the form $r_0+xg_{\alpha}$ for some polynomial $g_{\alpha}$ belonging to the $\C[x]$-submodule of $\C[x,z,t]$ generated by the monomials  $z^it^j$ with $0\leq i\leq k-2$ and $0\leq j\leq \ell-2$. Every such homomorphism $\alpha$ thus determines a threefold $\mathfrak{X}_{\alpha}=X(g_\alpha)$ defined by the equation $x^dy+r_0+xg_{\alpha}=0$.

\begin{theorem} \label{MainTh} With the notation above, the following statements hold:
\begin{itemize}
\item[$1)$]\  For every $g\in\C[x,z,t]$, there exists a homomorphism $\alpha:B \rightarrow \C[x]$ such that $X(g)$ is isomorphic to $\mathfrak{X}_{\alpha}$.
\item[$2)$]\  Two homomorphisms $\alpha,\beta:B\rightarrow \C[x]$ determine isomorphic threefolds $\mathfrak{X}_{\alpha}$ and $\mathfrak{X}_{\beta}$ if and only if
there exist constants $\lambda,\mu\in\C^*$ such that the following diagram commutes \[\xymatrix{ B \ar[rr]^-{\pi_d\circ \alpha} \ar[d]_{\zeta_{\lambda}} & & \C[x]/(x^d) \ar[d]^{x\mapsto \mu x} \\ B \ar[rr]^-{\pi_d\circ \beta} & & \C[x]/(x^d)}\] where $\pi_d:\C[x]\rightarrow \C[x]/(x^d)$ denotes the natural projection and where $\zeta_{\lambda}:B\rightarrow B$ is the linear automorphism defined by $a_{i,j}\mapsto \lambda^{\ell i+kj-k\ell} a_{i,j}$.
\end{itemize}
\end{theorem}

\begin{remark} Noting that the subvariety $\mathcal{V}$ of $T\times \ba^2=\mathrm{Spec}(B[z,t])$ with equation $F=0$ is the (mini)-versal deformation of the curve $C_0\subset \ba^2$ with equation $r_0=0$ (see e.g. \S 14.1 in \cite{Har}), we can re-interpret the above result as the fact that for fixed $d\geq 2$, isomorphism classes of hypersurfaces of the form $X(g)$ are in one-to-one correspondence with one-parameter infinitesimal embedded deformations of order $d-1$ of $C_0$, up to the equivalence between such deformations defined in the second assertion of the theorem.
\end{remark}

\begin{proof}
1) In view of Lemma \ref{rem}, to prove the first assertion, it is enough to
show that, given a polynomial  $g\in\mathbb{C}[x,z,t]$, there exist elements $b_{m}$,
$m=1,\ldots,d-1$ in the sub-vector space of $\mathbb{C}[z,t]$ generated
by the monomials $z^{i}t^{j}$, $0\leq i\leq k-2$, $0\leq j\leq\ell-2$
and a $\mathbb{C}[x]$-automorphism $\varphi$ of $\mathbb{C}[x][z,t]$
which maps the ideal $(x^{d},r_{0}+xg)$ into the ideal $(x^{d},r_{0}+\sum_{m=1}^{d-1}b_{m}x^{m})$.
We may write
\[
g=\sum_{m\geq1}u_{m}x^{m-1}=\sum_{m\geq1}(s_{m}+t_{m})x^{m-1}
\]
where for every $m$, $s_{m}$ belongs to the sub-vector space of
$\mathbb{C}[z,t]$ generated by the monomials $z^{i}t^{j}$, $0\leq i\leq k-2$,
$0\leq j\leq\ell-2$, while $t_{m}$ is in the ideal $(z^{k-1},t^{\ell-1})\mathbb{C}[z,t]$.
Let $t_0=0$ and denote by $\nu=\nu(g)$ the maximal integer with the property that $t_{m}=0$
for every $m\leq\nu-1$. If $\nu=d$ then we are done. Otherwise, we will proceed by induction.

Note that the image of the $\mathbb{C}$-derivation of $\mathbb{C}[z,t]/(r_{0})$
induced by the Jacobian derivation $\mathrm{Jac}(r_{0},\cdot)$ of
$\mathbb{C}[z,t]$ is equal to the ideal of $\mathbb{C}[z,t]/(r_{0})$
generated by the residue classes of $z^{k-1}$ and $t^{\ell-1}$. Indeed, one checks for example that, given $a\geq k-1$ and $b\geq0$, we have $$\mathrm{Jac}(r_0,\lambda z^{a-k+1}t^{b+1})=z^at^b-\lambda \ell(a-k+1)z^{a-k}t^br_0\equiv z^at^b\mod(r_0),$$
 where $\lambda=(k(b+1)+\ell(a-k+1))^{-1}$. This fact guarantees
the existence of polynomials $h,f\in\mathbb{C}[z,t]$ such that $\mathrm{Jac}(h,r_{0})=t_{\nu}+r_{0}f$.
Let us consider the $\mathbb{C}[x]/(x^{\nu+1})$-automorphism $\overline{\varphi}=\exp(-x^{\nu}\mathrm{Jac}(h,\cdot))$
of $\mathbb{C}[x]/(x^{\nu+1})[z,t]$. Since its Jacobian is a nonzero constant, it lifts, by the main theorem  in \cite{EMV}, to a $\mathbb{C}[x]$-automorphism $\varphi$ of $\mathbb{C}[x][z,t]$ such that $\varphi(x)=x$ and $\varphi(a)\equiv a-x^{\nu}\mathrm{Jac}(h,a)\mod(x^{\nu+1})$ for all $a\in\mathbb{C}[x,z,t]$. Then,  we have the following congruences modulo $(x^{\nu+1})$.
\begin{align*}
\varphi(r_{0}+xg) & \equiv  r_{0}+\sum_{m=1}^{\nu-1}s_{m}x^{m}+x^{\nu}(s_{\nu}+t_{\nu})-x^{\nu}\mathrm{Jac}(h,r_{0}) &\mod(x^{\nu+1})\\
  & \equiv r_{0}+\sum_{m=1}^{\nu}s_{m}x^{m}-x^{\nu}r_{0}f &\mod(x^{\nu+1})\\
  & \equiv (1-x^{\nu}f)(r_{0}+\sum_{m=1}^{\nu}s_{m}x^{m}) &\mod(x^{\nu+1}).
\end{align*}
It follows that there exists a polynomial $R\in\mathbb{C}[x,z,t]$ such that $\varphi(r_{0}+xg)\equiv(1-x^{\nu}f)(r_{0}+\sum_{m=1}^{\nu}s_{m}x^{m}+x^{\nu+1}R) \mod(x^{d})$. Letting $\widetilde{g}=\sum_{m=1}^{\nu}s_{m}x^{m-1}+x^{\nu}R$, we obtain that $\varphi$ maps the ideal $(x^{d},r_{0}+xg)$ into the ideal $(x^{d},r_{0}+x\widetilde{g})$,
and since $\nu(\tilde{g})>\nu(g)$ by construction, we are done by induction.
\\

2) Let us first rephrase the second assertion. If we let $\alpha(a_{i,j})=x\alpha_{i,j}(x)$ and $\beta(a_{i,j})=x\beta_{i,j}(x)$ with $\alpha_{i,j}, \beta_{i,j}\in\mathbb{C}[x]$ for $0\leq i\leq k-2, 0\leq j\leq \ell -2$, the latter states that the threefolds $\mathfrak{X}_{\alpha}$
and $\mathfrak{X}_{\beta}$ are isomorphic if and only if there exist two constants $\lambda,\mu\in\C^*$ such that
$$\mu x\alpha_{i,j}(\mu x)\equiv\lambda^{\ell i+kj-k\ell}x\beta_{i,j}(x) \mod(x^d)$$
 for all $0\leq i\leq k-2$ and all $0\leq j\leq \ell -2$.  If such constants $\lambda,\mu\in\mathbb{C}^{*}$
exist then the automorphism $\varphi$ of $\mathbb{C}[x,z,t]$ defined
by $\varphi(x)=\mu x$, $\varphi(z)=\lambda^{-\ell}z$, $\varphi(t)=\lambda^{-k}t$ satisfies that $\varphi(r_0+xg_{\alpha})\equiv\lambda^{-k\ell}(r_0+xg_{\beta})$ modulo $(x^{d})$. Thus, $\varphi$ maps the ideal $(x^{d},r_{0}+xg_{\alpha})$ into the ideal $(x^{d},r_{0}+xg_{\beta})$, and $\mathfrak{X}_{\alpha}$ and $\mathfrak{X}_{\beta}$ are isomorphic by Lemma \ref{rem}.

Conversely, suppose that $\mathfrak{X}_{\alpha}$
and $\mathfrak{X}_{\beta}$ are isomorphic and let $\varphi$ be an automorphism of $\mathbb{C}[x,z,t]$ which fixes the ideal $(x)$ and maps the ideal $(x^{d},r_{0}+xg_{\alpha})$ into the ideal $(x^{d},r_{0}+xg_{\beta})$.  Up to changing $\mathfrak{X}_{\beta}$
by its image under an isomorphism coming from an automorphism of $\mathbb{C}^3$ of the type $(x,z,t)\mapsto(\mu x,\lambda^{-\ell}z,\lambda^{-k}t)$ as above, we
may further assume that $\varphi$ is a $\mathbb{C}[x]$-automorphism
of $\mathbb{C}[x][z,t]$ which is the identity modulo $(x)$ and we are thus reduced to the following statement. ``\emph{Suppose that there exists a $\C[x]$-automorphism $\varphi$ of $\C[x][z,t]$ which is congruent to the identity modulo $(x)$ and such that $\varphi(r_0+xg_\alpha)\in (x^d,r_0+xg_\beta)$. Then $g_\alpha$ and $g_\beta$ are congruent modulo $(x^{d-1})$}.''

Choose $\nu$ maximal such that $\varphi$ is congruent to the identity modulo $(x^\nu)$.  If $\nu\geq d$ then we are done. So, suppose that $\nu\leq d-1$. Writing down that the Jacobian of  $\varphi$ is constant equal to 1, we remark that there exists a polynomial $h\in\mathbb{C}[z,t]$ such that $\varphi(z)$ and $\varphi(t)$ are congruent modulo $(x^{\nu+1})$ to $z+x^{\nu}h_t$ and to $t-x^{\nu}h_z$, respectively. Consequently, we have $\varphi(r_{0}+xg_{\alpha})\equiv r_{0}+xg_{\alpha}+x^{\nu}\mathrm{Jac}(r_{0},h)$ modulo $(x^{\nu+1})$. On the other hand, since $\varphi(r_{0}+xg_{\alpha})\in(x^{d},r_{0}+xg_{\beta})$ and since, by definition,  $g_{\alpha}$ and $g_{\beta}$  contain no monomials of the form $cx^{i_1}z^{i_2}t^{i_3}$ with $i_2\geq k-1$ or $i_3\geq\ell-1$, there exists  a polynomial $a\in\mathbb{C}[z,t]$ such that $\varphi(r_{0}+xg_{\alpha})\equiv r_0+xg_{\beta}+x^{\nu}ar_0$ modulo $(x^{\nu+1})$.

Writing $xg_{\alpha}=\sum_{m\geq1}s_{\alpha,m}x^{m}$
and $xg_{\beta}=\sum_{m\geq1}s_{\beta,m}x^{m}$ as  before, we have thus
$$ r_{0}+\sum_{m=1}^{\nu-1}s_{\alpha,m}x^{m} +x^{\nu}(s_{\alpha,\nu}+\mathrm{Jac}(r_{0},h)) \equiv  r_{0}+\sum_{m=1}^{\nu-1}s_{\beta,m}x^{m} +x^{\nu}(s_{\beta,\nu}+ar_0)$$ modulo $(x^{\nu+1})$.
Since $\mathrm{Jac}(r_{0},h)$ and $r_{0}$ both belong to the ideal $(z^{k-1},t^{\ell-1})$ of $\mathbb{C}[z,t]$,
we conclude that $s_{\alpha,m}=s_{\beta,m}$ for every $m=1,\ldots,\nu$
and that $\mathrm{Jac}(r_{0},h)\in r_{0}\mathbb{C}[z,t]$. This implies
in turn that $h=\gamma r_0+c$ for some $\gamma\in\mathbb{C}[z,t]$
and $c\in\mathbb{C}$. Now consider the $\mathbb{C}[x]/(x^{d})$-automorphism
$\overline{\theta}=\exp(\delta)$ of $\mathbb{C}[x]/(x^{d})[z,t]$
associated with the derivation $\delta=x^{\nu}\mathrm{Jac}(\cdot,\gamma(r_{0}+xg_{\alpha}))$.
Since $\overline{\theta}$ has Jacobian determinant equal to $1$ (see \cite{mj09}), we deduce again from \cite{EMV}
that it lifts to a $\mathbb{C}[x]$-automorphism $\theta$ of $\mathbb{C}[x][z,t]$.
By construction $\theta\equiv\varphi$ modulo $x^{\nu+1}$ and since
$r_{0}+xg_{\alpha}$ divides $\delta(r_{0}+xg_{\alpha})=x^{\nu}(r_{0}+xg_{\alpha})\mathrm{Jac}(r_{0}+xg_{\alpha},\gamma)$,
it maps the ideal $(x^{d},r_{0}+xg_{\alpha})$ into itself. Thus $\varphi_{1}=\varphi\circ\theta^{-1}$
is a $\mathbb{C}[x]$-automorphism of $\mathbb{C}[x][z,t]$ congruent
to identity modulo $x^{\nu+1}$ which maps $(x^{d},r_{0}+xg_{\alpha})$
into $(x^{d},r_{0}+xg_{\beta})$ and we can conclude the proof  by induction.
\end{proof}

\section{ An example of a non-extendible automorphism}

Most of the algebraic results of the previous two sections can be generalized to the situation where $r_0\in \C[z,t]$ defines a connected and reduced plane curve, provided that the Makar-Limanov and Derksen invariants of the corresponding threefolds are equal to $\C[x]$ and $\C[x,z,t]$ respectively. Here we illustrate a new phenomenon in a particular case where the zero set of $r_0$ is not connected.

For this section, we fix $d=2$,  $r_0=z(zt^2+1)$ and we let $X$ be the smooth hypersurface in $\ba^4=\mathrm{Spec}(\C[x,y,z,t])$ defined by the equation $$ P=x^2y+z(zt^2+1)=0.$$
Note that $X$ is neither factorial nor topologically contractible. It is straightforward to check using the methods in \cite{K-ML2} that $\mathrm{Dk}(X)=\C[x,z,t]$ and $\mathrm{ML}(X)=\C[x]$. We will use the fact that the plane curve $C_0$ with equation $r_0=z(zt^2+1)=0$ has two connected components to construct a particular automorphism $\tilde{\varphi}$ of $X$ which cannot extend to the ambient space $\ba^4$.

Starting from the  polynomial $h=zt^2\in\C[z,t]$, we first construct in a similar way as in \cite{mj09} an automorphism $\varphi$ of $\C[x,z,t]$ as follows: Noting that the $\C[x]/(x^2)$-automorphism $\overline\varphi=\exp(x\Jac(h,\cdot))$ of $\C[x]/(x^2)[z,t]$ has Jacobian equal to $1$, we deduce again from \cite{EMV}, that it lifts to a $\C[x]$-automorphism $\varphi$ of $\C[x,z,t]$. In other words, there exists an automorphism $\varphi$ of $\C[x,z,t]$  with $\varphi(x)=x$ and such that $\varphi\equiv\overline\varphi \mod (x^2)$.
One checks further using the explicit formulas  that $\overline\varphi(z)=(1-2tx)z$,  $\overline\varphi(t)=(1+tx)t$ and that $\overline{\varphi}(r_0)=(1-2xt)r_0$. So $\varphi$ preserves the ideals $(x)$ and $J=(x^2,z(zt^2+1))$ and hence lifts to a $\C[x]$-automorphism $\tilde{\varphi}$ of the coordinate ring $\C[X]$ of $X$.

\begin{proposition} The automorphism of $X\subset \ba^4$ determined by $\tilde{\varphi}$ cannot extend to an automorphism of the ambient space.
\end{proposition}

\begin{proof}
Suppose by contradiction that there exists an automorphism $\Phi$ of $\C[x,y,z,t]$ which extends $\tilde{\varphi}$. Then there exists a constant $\lambda\in\C^*$ such that $\Phi(P)=\lambda P$ and so we obtain a commutative diagram \[\xymatrix { \C[x,y,z,t] \ar[r]^{\Phi} & \C[x,y,z,t] \\ \C[u] \ar[u] \ar[r]^{u\mapsto \lambda u} & \C[u] \ar[u]}\] where $\C[u]\rightarrow \C[x,y,z,t]$ maps $u$ onto $P=x^2y+z(zt^2+1)\in \C[x,y,z,t]$. Passing to the field of fractions of $\C[u]$, we see that $\Phi$ induces an isomorphism between the $\C(u)$-algebras $$A=\C(u)[x,y,z,t]/(P-u) \quad \textrm{and} \quad A_\lambda=\C(u)[x,y,z,t]/(P-\lambda^{-1}u).$$ Now the key observation is that the Makar-Limanov and Derksen invariants of $A$ and $A_{\lambda}$ are equal to $\C(u)[x]$ and $\C(u)[x,z,t]$ respectively (this follows from an application of Theorem 9.1 of \cite{K-ML2} which in fact only requires that the base field has characteristic 0, we do not give the details here). Now since $\Phi$ induces an isomorphism between the Makar-Limanov invariants of $A$ and $A_{\lambda}$, it follows that  $\Phi(x)=\mu x+\nu$ for some  $\mu\in\C(u)^*$ and $\nu\in\C(u)$. From the fact that the ideal $(x,P)$ of $\C[x,y,z,t]$ is not prime, whereas the ideals $(x+c,P)$ are prime for all $c\in\C^*$, we can conclude that $\nu=0$ and since $\tilde{\varphi}(x)=x$, we eventually deduce that $\mu=1$. Thus $\Phi(x)=x$. Next, noting that $\Phi$ induces an automorphism $\overline{\Phi}$ of $\C[x,y,z,t]/(x)\cong\C[y,z,t]$ with $\overline{\Phi}(z(zt^2+1))=\lambda z(zt^2+1)$, we deduce that there exists $\alpha\in\C^*$ such that $\overline{\Phi}(z)=\alpha z$. By comparing with $\tilde{\varphi}(z)$, we find that $\alpha=1$. Thus  $\overline{\Phi}(zt^2+1)=z\overline{\Phi}(t)^2+1=\lambda(zt^2+1)$ and by considering the constant terms in the last equality, we conclude that in fact $\lambda=1$. Also, since $\overline{\varphi}(t)\equiv t \mod (x)$, we have that $\overline{\Phi}(t)\equiv t \mod (x)$. Thus,  $\Phi$ is congruent to the identity modulo $(x)$.

So $\Phi$ actually induces a $\C(u)[x]$-automorphism of $A$ and the observation made on the Makar-Limanov and Derksen invariants of $A$ implies that the induced automorphism preserves the subring $\C(u)[x,z,t]$ of $A$ and fixes the ideal $(x^2,r_0-u)$. This implies that there exists $H\in \C[u][z,t]$ such that $\Phi(f)\equiv f+xJac(H,f)$ mod $x^2$ for every $f\in \C(u)[z,t]$. Furthermore, since $\Phi$ is an extension of $\tilde{\varphi}$, the residue class of $H$ in $\C[u][z,t]/(u)\simeq \C[z,t]$ coincides with $h+c$ for some $c\in \C$. But on the other hand, one has $Jac(H,r_0-u)\in (r_0-u)$ as the restriction of $\Phi$ fixes the ideal $(x^2,r_0-u)$. Since $r_0-u$ is irreducible, this implies that $H$ is, up to the addition of a polynomial in $\C[u]$, an element of the ideal $(r_0-u)\C[u][z,t]$ and so its image in $\C[u][z,t]/(u)\simeq \C[z,t]$ is a regular function on $\ba^2$ whose restriction to the curve $C_0=\{r_0=0\}$ is constant. This is absurd since by construction $h+c=zt^2+c$ is locally constant but not constant on $C_0$.
\end{proof}

\begin{remark} Note that there are many examples of non-extendible automorphisms of hypersurfaces. In this setting, for example,
we showed in \cite{d-mj-p} that the  hypersurface in $\ba^4$ given by the equation $$x^2y+z^2+t^3+x(1+x+z^2+t^3)=0,$$ which is in fact isomorphic to the Russell cubic as a variety, admits automorphisms that do not extend to the ambient space. However, none of these non-extendible automorphisms  fixes $x$. The present example exhibits a new phenomenon coming from the fact that the curve defined by $r_0=0$ is not connected.
Note also that viewing $X$ as a subvariety of $\ba^1\times \ba^3=\mathrm{Spec}(\C[x][y,z,t])$ the restriction of the automorphism determined by $\tilde{\varphi}$ to every  fiber $X_s$, $s\in \ba^1$ of the first projection $\mathrm{pr}_x:X\rightarrow \ba^1$ actually extend to an automorphism of $\ba^3_s$: indeed, $\tilde{\varphi}$ restricts to the identity modulo $x$ and therefore the restriction of the corresponding automorphism to $X_0$ extends. On the other hand, the argument used in the proof of Lemma \ref{rem} shows immediately that the induced automorphism of $X\mid_{\ba^1\setminus\{0\}}$ extends to an automorphism of $\ba^1\setminus\{0\}\times \ba^3$.
\end{remark}

\end{document}